\documentclass{article}

\usepackage{amssymb}
\usepackage{amsmath,amsthm}
\usepackage[dvips]{graphicx}
\usepackage{wrapfig}
\usepackage{bm}

\newtheorem{theorem}{Theorem}[section]
\newtheorem{corollary}[theorem]{Corollary}
\newtheorem{lemma}[theorem]{Lemma}
\newtheorem{proposition}[theorem]{Proposition}

\theoremstyle{definition}

\newtheorem{example}[theorem]{Example}
\theoremstyle{remark}

\begin{document}
\title{Region crossing change on spatial theta-curves}

\author{
Ayaka Shimizu \thanks{Department of Mathematics, National Institute of Technology, Gunma College, 580 Toriba-cho, Maebashi-shi, Gunma, 371-8530, Japan. Email: shimizu@nat.gunma-ct.ac.jp }
\and 
Rinno Takahashi \thanks{Department of Information and Computer Engineering, National Institute of Technology, Gunma College, 580 Toriba-cho, Maebashi-shi, Gunma, 371-8530, Japan.  }
}
\date{\today}

\maketitle

\begin{abstract}
A region crossing change at a region of a spatial-graph diagram is a transformation changing every crossing on the boundary of the region. 
In this paper, it is shown that every spatial graph consisting of theta-curves can be unknotted by region crossing changes. 
\end{abstract}

\section{Introduction} 

A {\it knot} is an embedding of a circle in $S^3$, and a {\it link} is an embedding of some circles in $S^3$. 
A {\it spatial graph} is an embedding of a graph in $S^3$. 
A {\it diagram} of a knot, link or spatial graph $G$ is a projection of $G$ to $S^2$ with over/under information at each crossing, where each crossing is made of two arcs intersecting transversely. 
A knot, link or spatial graph $G$ is {\it unknotted}, or {\it trivial}, if $G$ has a diagram which has no crossings. 
It is well-known that any diagram of a knot or link can be transformed into a diagram of a trivial knot or link by a finite number of {\it crossing changes}, where a crossing change is a local transformation shown in Figure \ref{cc}\footnote{Some spatial graphs cannot be unknotted by crossing changes, such as the complete graph $K_5$. }. 
\begin{figure}[ht]
\begin{center}
\includegraphics[width=30mm]{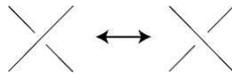}
\caption{Crossing change. }
\label{cc}
\end{center}
\end{figure}

A {\it region crossing change} at a region of a diagram of a knot, link or spatial graph is a local transformation which yields crossing changes at all the crossings on the boundary of the region. 
For knots, the following theorem is shown: 

\phantom{x}
\begin{theorem}[\cite{shimizu-rcc}]
Any diagram of any knot can be transformed into a diagram of the trivial knot by a finite number of region crossing changes. 
\label{thm11}
\end{theorem}
\phantom{x}

\noindent For links, the following is shown: 

\phantom{x}
\begin{theorem}[\cite{cheng}]
Any diagram of a link $L$ can be transformed into a diagram of a trivial link by a finite number of region crossing changes if and only if $L$ is a proper link. 
\label{cheng-thm}
\end{theorem}
\phantom{x}

\noindent For knots and links, the unknottability on region crossing change does not depend on the choice of diagram. 
On the other hand, for spatial graphs, it depends on the choice of diagram. 
For the two diagrams representing the same spatial graph in Figure \ref{diagram-dependent}, the left one cannot be unlinked by region crossing changes at any set of regions, whereas the right one gets unlinked by a region crossing change at the shaded region. \\
\begin{figure}[ht]
\begin{center}
\includegraphics[width=35mm]{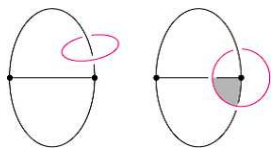}
\caption{Unknottability on region crossing change is diagram-dependent for spatial graphs. }
\label{diagram-dependent}
\end{center}
\end{figure}

\noindent For a spatial graph with a single component, i.e., an embedding of a connected graph, region crossing change is studied in \cite{HSS}. 
In this paper, we study region crossing change on a spatial graph with some components, and show the following theorems. 

\phantom{x}
\begin{theorem}
Let $G$ be a two-component spatial graph consisting of a spatial $\theta$-curve\footnote{Spatial $\theta$-curve is explained in Section \ref{spatial-graphs}. } and a knot. 
There exists a diagram $D$ of $G$ such that $D$ can be unknotted by a finite number of region crossing changes. 
\label{theta-knot}
\end{theorem}

\phantom{x}
\begin{theorem}
Let $n$ be a positive integer. 
Let $G$ be an $n$-component spatial graph whose components are all spatial $\theta$-curves. 
There exists a diagram $D$ of $G$ such that $D$ can be unknotted by a finite number of region crossing changes. 
\label{theta-n}
\end{theorem}
\phantom{x}

\noindent The rest of the paper is organized as follows: 
In Section \ref{spatial-graphs}, we review the study of region crossing change on spatial graphs. 
In Section \ref{proofs}, we prove Theorems \ref{theta-knot} and \ref{theta-n}. 
In Section \ref{s-handcuff}, we also consider spatial handcuff graphs. 
In Section \ref{incidence-matrices}, we study incidence matrices for spatial-$\theta$-curve diagrams. 
In Section \ref{ineffective-sets}, we consider ineffective sets for spatial-$\theta$-curve diagrams.

\section{Spatial-graph diagrams} 
\label{spatial-graphs} 

In this section, we prepare some terms of spatial-graph diagrams regarding that knots and links are included by spatial graphs, and review some results on region crossing change. \\

A {\it graph} is a pair of sets of vertices and edges. 
Each connected graph $g$ which has at least one vertex has a {\it maximal tree}, where a maximal tree is a connected subgraph of $g$ which includes no cycles and includes all the vertices of $g$. 
Since a spatial graph is an embedding of a graph, every connected spatial graph which has a vertex has a maximal tree. \\

It is known that two diagrams represent the same spatial graph if and only if they are equivalent up to the {\it Riedemeister moves} shown in Figure \ref{reidemeister}. 
\begin{figure}[ht]
\begin{center}
\includegraphics[width=90mm]{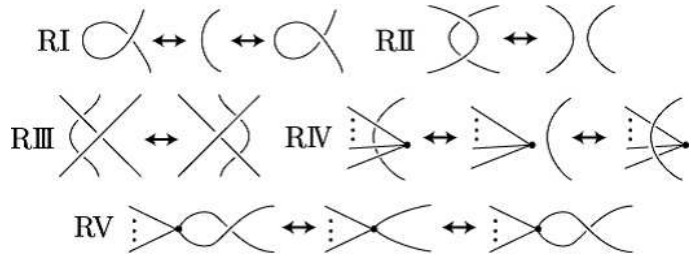}
\caption{Reidemeister moves. }
\label{reidemeister}
\end{center}
\end{figure}
A crossing $p$ of a spatial-graph diagram $D$ is {\it reducible} if one can draw a circle $C$ on $S^2$ such that $C$ intersects only $p$ transversely as shown in Figure \ref{reducible}. 
A diagram $D$ is said to be {\it reducible} if $D$ has a reducible crossing. 
Otherwise, $D$ is said to be {\it reduced}, or {\it irreducible}.  
\begin{figure}[ht]
\begin{center}
\includegraphics[width=40mm]{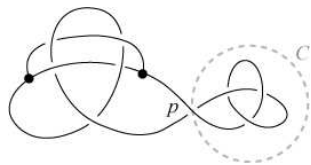}
\caption{A reducible diagram with a reducible crossing $p$. }
\label{reducible}
\end{center}
\end{figure}
A {\it cutting circle} of a diagram $D$ is a circle on $S^2$ intersecting an edge transversely at exactly one point as shown in Figure \ref{cutting}. 
We call such an edge a {\it cutting edge}. 
\begin{figure}[ht]
\begin{center}
\includegraphics[width=45mm]{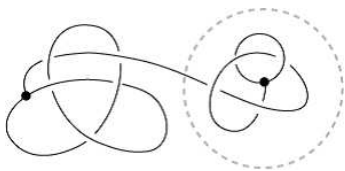}
\caption{A diagram with a cutting edge. }
\label{cutting}
\end{center}
\end{figure}
For spatial graphs of one component, the following is shown: 

\phantom{x}
\begin{theorem}[\cite{HSS}]
Let $G$ be a spatial graph of one component, and let $D$ be a diagram of $G$ without cutting edges. 
Any crossing change on $D$ is realized by a finite number of region crossing changes. 
\label{HSSthm}
\end{theorem}
\phantom{x}

\noindent A {\it $\theta$-curve} is a connected graph consisting of two vertices $v_1$ and $v_2$ and three edges which are adjacent to both $v_1$ and $v_2$. 
A {\it spatial $\theta$-curve} is an embedding of a $\theta$-curve in $S^3$. 
Since any diagram of a spatial $\theta$-curve does not have a cutting edge, we have the following: 

\phantom{x}
\begin{lemma}[\cite{HSS}]
Any crossing change on any diagram of a spatial $\theta$-curve is realized by a finite number of region crossing changes. 
\label{thHSS}
\end{lemma}
\phantom{x}

\noindent Theorem \ref{HSSthm} is a generalization of the following: 

\phantom{x}
\begin{lemma}[\cite{shimizu-rcc}]
Any crossing change on any knot diagram is realized by a finite number of region crossing changes. 
\label{realize-k}
\end{lemma}
\phantom{x}

\noindent Theorem \ref{thm11} is obtained from Lemma \ref{realize-k} by repeating such region crossing changes. 
Similarly, we can unknot any diagram of a spatial $\theta$-curve by a finite number of region crossing changes by Lemma \ref{thHSS}.

\section{Proofs of Theorems \ref{theta-knot} and \ref{theta-n}} 
\label{proofs} 

In this section we prove Theorems \ref{theta-knot} and \ref{theta-n}. 

\phantom{x} 
\noindent {\it Proof of Theorem \ref{theta-knot}}. \ 
For a spatial graph $G$ consisting of a spatial $\theta$-curve and a knot, take a maximal tree $T$ of the $\theta$-curve component. 
Then $G$ has a diagram $D_0$ such that the corresponding part of $T$ has no crossings (see, for example, \cite{kawauchi}). 
We call the vertices $v_1$, $v_2$ and edges $e_1$, $e_2$ and $e_3$ of the $\theta$-curve component as indicated on $D_0$ in Figure \ref{d0}. 
We note that the maximal tree $T$ consists of $v_1$, $v_2$ and $e_3$ without crossings, and that $e_1$ and $e_2$ may have crossings. \\
\begin{figure}[ht]
\begin{center}
\includegraphics[width=30mm]{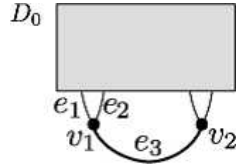}
\caption{A diagram $D_0$ of $G$ without crossings on $e_3$. }
\label{d0}
\end{center}
\end{figure}

Shrink $e_1$ so that $v_1$ moves to the right side, near $v_2$, and that $e_2$ and $e_3$ follow $e_1$ making a sufficiently narrow wheel track. 
We call the result $D_1$. 
Next, similarly shrink $e_2$ so that $v_2$ moves to the left side, and that $e_1$ and $e_3$ follow $e_2$ making a wheel track, and obtain $D_2$. 
On $D_2$, the $\theta$-curve component and the knot component make pairwise crossings at the wheel tracks.\\

We can change any such crossing pair by region crossing changes as follows: 
Let $p$ be a crossing pair. 
Let $t_i$ be the wheel track where $p$ belongs, and let $v_i$ be the adjacent vertex of $t_i$. 
Apply region crossing changes at the regions in the wheel track $t_i$ in order from $v_i$ to $p$. 
Then only the crossing pair $p$ changes. \\

Apply region crossing changes on $D_2$ so that the $\theta$-curve component is over than the knot component at every crossing between them. 
An example is shown in Figure \ref{d012}. 
We call the result $D_3$. \\
\begin{figure}[ht]
\begin{center}
\includegraphics[width=100mm]{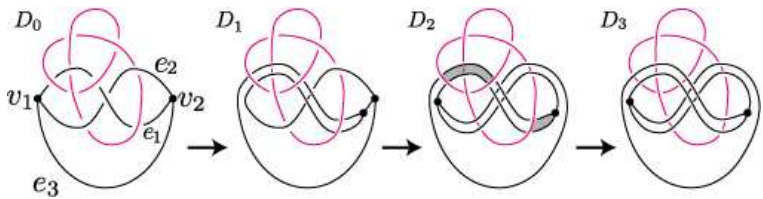}
\caption{Take a diagram $D_0$ so that $e_3$ has no crossing. 
Shrink $e_1$ so that $v_1$ moves to the right side, and obtain $D_1$. 
Shrink $e_2$ so that $v_2$ moves to the left side, and obtain $D_2$. 
On $D_2$, we can lift up the $\theta$-curve component over than the knot component by region crossing changes at some regions inside the wheel tracks. }
\label{d012}
\end{center}
\end{figure}

Let $D_3^{\theta}$ be the $\theta$-curve component diagram in $D_3$. 
Recall that $D_3^{\theta}$ has a set of regions $S^{\theta}$ such that $D_3^{\theta}$ is transformed into a diagram of the trivial $\theta$-curve by region crossing changes at $S^{\theta}$ (Lemma \ref{thHSS}). 
Apply region crossing changes on $D_3$ at the regions $S_3^{\theta}$ of $D_3$ corresponding to the regions in $S^{\theta}$. 
Then, the $\theta$-curve component gets unknotted, while crossings between the components are unchanged because there are disjoint four regions around each crossing between them, and an even number of them belongs to $S_3^{\theta}$. \\

Thus we obtain a diagram $D_4$, representing a splittable spatial graph with the $\theta$-curve component unknotted. 
Let $D_4^k$ be the knot component diagram in $D_4$. 
Recall that $D_4^k$ has a set of regions $S^k$ such that $D_4^k$ is transformed into a diagram $O_4^k$ of the trivial knot by region crossing changes at $S^k$ (Lemma \ref{realize-k}). 
Apply region crossing changes to $D_4$ at the regions of $D_4$ corresponding to the regions in $S^k$. 
Then $D_4^k$ gets unknotted, and crossings between the components are unchanged. 
Remark that some reducible crossings of $D_4^k$ may be different from $O_4^k$ after the region crossing changes while non-reducible crossings are the same. 
This does not matter since the unknottedness is unchanged even if a reducible crossing is changed. 
Similarly, the unknottedness of the $\theta$-curve component is also unchanged. 
Thus, we obtain a diagram of an unknotted spatial graph by a finite number of region crossing changes from a diagram $D_2$ of $G$. 
\hfill$\Box$

\phantom{x}
\noindent Next, we prove Theorem \ref{theta-n} in a similar way. 

\phantom{x}
\noindent {\it Proof of Theorem \ref{theta-n}}. \ 
Let $G$ be a spatial graph consisting of $n$ spatial $\theta$-curves, and let $D$ be a diagram of $G$. 
Take a maximal tree for each $\theta$-curve component. 
Gather the $n$ maximal trees to the same place in the following way. 
Choose a region $R$ of $D$. 
Take a small part of an edge of each maximal tree, and move it into $R$ by Reidemeister moves. 
Then move the adjacent vertices into $R$ along the edge, by Reidemeister moves. 
We call the result $D_0$, and name the components ${\theta}^k$, vertices $v^k _i$ and edges $e^k _j$ as shown in Figure \ref{d0n}. \\
\begin{figure}[ht]
\begin{center}
\includegraphics[width=70mm]{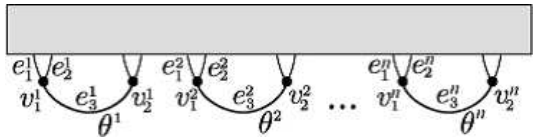}
\caption{Diagram $D_0$. }
\label{d0n}
\end{center}
\end{figure}

Move $v^i_1$ to another side near $v^i _2$ by shrinking $e^i _1$ for each $i$ with an order, where $e^i _2$ and $e^i _3$ follow $e^i _1$ making a wheel track, and we obtain a diagram $D_1$. 
Then move $v^i_2$ to another side by shrinking $e^i _2$ for each $i$, where $e^i _1$ and $^i _3$ follow $e^i _2$ making a wheel track. 
Thus, we obtain a diagram $D_2$ of $G$, and this is the diagram we required. 

We can transform $D_2$ into a diagram such that ${\theta}^i$ is over than ${\theta}^j$ ($i < j$) at each crossing between them by region crossing changes using the wheel-track method of the proof of Theorem \ref{theta-knot}. 
We call such a diagram $D_3$. 
Each diagram of ${\theta}^i$ has a set $S^i$ of regions such that ${\theta}^i$ gets unknotted by the region crossing changes by Lemma \ref{thHSS}. 
Apply region crossing changes at the corresponding regions in order from $S^1$ to $S^n$. 
Then we obtain a diagram of the unknotted spatial graph. 
\hfill$\Box$

\phantom{x}
\noindent From Theorems \ref{cheng-thm}, \ref{theta-knot} and \ref{theta-n}, we have the following corollary:

\phantom{x}
\begin{corollary}
Let $G$ be a spatial graph consisting of some spatial $\theta$-curves and a proper link. 
There exists a diagram $D$ of $G$ such that $D$ can be unknotted by a finite number of region crossing changes. 
\end{corollary}

\section{Spatial handcuff graphs} 
\label{s-handcuff} 

A {\it handcuff graph} is a connected graph consisting of two vertices $v_1$, $v_2$ and two loops based on $v_1$ and $v_2$, and an edge connecting $v_1$ and $v_2$. 
A {\it spatial handcuff graph} is an embedding of a handcuff graph in $S^3$. 
Similarly to spatial $\theta$-curve, we can unknot a spatial graph consisting of a spatial handcuff graph and a knot by region crossing change by taking a suitable diagram as shown in Figure \ref{handcuff-unknotting}. 
\begin{figure}[ht]
\begin{center}
\includegraphics[width=90mm]{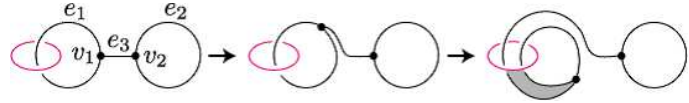}
\caption{A spatial-graph diagram which can be unlinked by region crossing changes by moving the diagram. }
\label{handcuff-unknotting}
\end{center}
\end{figure}
We have the following corollary: 

\phantom{x}
\begin{corollary}
Let $G$ be a two-component spatial graph consisting of a spatial handcuff graph and a knot. 
There exists a diagram $D$ of $G$ such that $D$ can be unknotted by a finite number of region crossing changes. 
\end{corollary}
\phantom{x}

\begin{proof}
Let $e_1$, $e_2$ be the loop edge based on $v_1$, $v_2$, respectively, and let $e_3$ be the non-loop edge of the handcuff component. 
Then $v_1$, $v_2$ and $e_3$ form a maximal tree of the handcuff component. 
Take a diagram $D$ of $G$ such that $e_3$ has no crossings. 
If the handcuff component in $D$ has a cutting edge, apply Reidemeister moves to $D$ so that there are no cutting edges, and call the result $D_0$. 
Move $v_1$ along $e_1$ until it backs to the initial position, where $e_1$ and $e_3$ makes a wheel track, and we call the result $D_1$. 
Similarly, move $v_2$ along $e_2$ and obtain a diagram $D_2$, and this is the diagram we required. 
The rest steps are same to the proof of Theorem \ref{theta-knot}. 
\end{proof}
\phantom{x}

\noindent We also have the following corollary: 

\phantom{x}
\begin{corollary}
Let $G$ be a spatial graph consisting of some spatial $\theta$-curves, some spatial handcuff graphs and a proper link. 
There exists a diagram $D$ of $G$ such that $D$ can be unknotted by a finite number of region crossing changes. 
\end{corollary}
\phantom{x}

\section{Incidence matrices} 
\label{incidence-matrices} 

In this section, we consider incidence matrices, and show the following: 

\phantom{x}
\begin{proposition}
Let $D$ be a diagram of a spatial $\theta$-curve, and let $D'$ be a diagram obtained from $D$ by crossing changes at some crossings. 
There exist exactly eight sets of regions of $D$ such that $D$ is transformed into $D'$ by region crossing changes at the regions. 
\label{eight}
\end{proposition}
\phantom{x}

\noindent We show an example: 

\phantom{x} 
\begin{example}
For the diagram $D$ in Figure \ref{t31}, if one wants to change the crossing $c_1$, one should solve the following simultaneous equations (see \cite{AS} for knots): 
\begin{figure}[ht]
\begin{center}
\includegraphics[width=30mm]{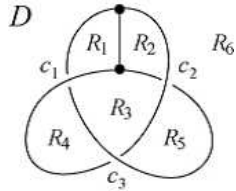}
\caption{A diagram of a spatial $\theta$-curve. }
\label{t31}
\end{center}
\end{figure}
\begin{align*}
\begin{cases}
c_1: \ x_1+x_3+x_4+x_6 \equiv 1 \pmod 2 \\
c_2: \ x_2+x_3+x_5+x_6 \equiv 0 \pmod 2 \\
c_3: \ x_3+x_4+x_5+x_6 \equiv 0 \pmod 2 \\
\end{cases}
\end{align*}
The first equation implies that one should choose an odd number of regions from $R_1$, $R_3$, $R_4$ and $R_6$ to change $c_1$. 
The second equation implies that one should choose an even number of regions from $R_2$, $R_3$, $R_5$ and $R_6$ not to change $c_2$. 
For the simultaneous equations, we have eight solutions, and then we have eight sets of regions $\{ R_1 \}$, $\{ R_5 , R_6 \}$, $\{ R_1 , R_2 , R_4 , R_6 \}$, $\{ R_2 , R_4 , R_5 \}$, $\{ R_1 , R_3 , R_6 \}$, $\{ R_3 , R_5 \}$, $\{ R_1 , R_2 , R_3 , R_4 \}$, $\{ R_2 , R_3 , R_4 , R_5 , R_6 \}$ to change only $c_1$ by region crossing changes. 
\label{ex-t31}
\end{example}
\phantom{x}

\noindent Let $G$ be a spatial graph, and let $D$ be a diagram of $G$. 
For crossings $c_1, c_2, \dots$ and regions $R_1, R_2, \dots$ of $D$, the {\it region choice matrix $M=(a_{ij})$ of $D$} is a matrix defined by the following: 
\begin{align*}
a_{ij}=
\begin{cases}
1 & \text{if } R_j \text{ is adjacent to } c_i \\
0 & \text{otherwise}
\end{cases}
\end{align*}
We note that the region choice matrix for knots is introduced in \cite{AS} not only for modulo 2. 
We also note that the region choice matrix of modulo 2 is the transposed matrix of the {\it incidence matrix} introduced in \cite{CG}. 
The region choice matrix of the diagram $D$ in Figure \ref{t31} is 
\begin{align*}
\begin{pmatrix}
1 & 0 & 1 & 1 & 0 & 1 \\
0 & 1 & 1 & 0 & 1 & 1 \\
0 & 0 & 1 & 1 & 1 & 1 
\end{pmatrix}
\end{align*}
and this is equivalent to the coefficient matrix of the simultaneous equation in Example \ref{ex-t31}. \\

For knots, the size of a region choice matrix of a knot diagram with $n$ crossings is $n \times (n+2)$, and then it is shown in \cite{CG} that the rank is $n$ using Lemma \ref{realize-k}, and the knot version of Proposition \ref{eight} is shown in \cite{kawauchi-ed} and \cite{hashizume} as the number of sets is four. 
We show Proposition \ref{eight} in the same way. 
First, we show the following: 

\phantom{x}
\begin{lemma}
Let $D$ be a diagram of a spatial $\theta$-curve with $n$ crossings. 
The size of a region choice matrix of $D$ is $n \times (n+3)$. 
\label{c3}
\end{lemma}
\phantom{x}

\noindent Since crossings correspond to the rows, and regions correspond to the columns, Lemma \ref{c3} follows from the following lemma: 

\phantom{x}
\begin{lemma}
Every diagram of a spatial $\theta$-curve has the number of regions three more than the number of crossings. 
\end{lemma}
\phantom{x}

\begin{proof}
Let $D$ be a diagram of a spatial $\theta$-curve with $n$ crossings, and let $|D|$ be a graph obtained from $D$ by regarding each crossing to be a vertex. 
That is, $|D|$ is a graph on $S^2$ with $v=n+2$ vertices. 
Looking locally at each vertex of $|D|$ which corresponds to a crossing on $D$, there are four edges around it, and looking at each vertex of $|D|$ which corresponds to a vertex on $D$, there are three edges around it. 
Hence, the number of total endpoints of edges of $|D|$ is $4n+6$. 
Since each edge has two endpoints, the number $e$ of edges of $|G|$ is $2n+3$. 
By substituting to the equation $v-e+f=2$ of Euler's characteristic of $S^2$, we have the number of the regions, $f=n+3$. 
\end{proof}
\phantom{x}

\noindent Secondary, we show the following lemma: 

\phantom{x}
\begin{lemma}
Let $D$ be a diagram of a spatial $\theta$-curve with $n$ crossings, and let $M$ be a region choice matrix of $D$. 
Then, the rank of $M$ is $n$, namely, $M$ is full-rank. 
\label{fullrank}
\end{lemma}

\begin{proof}
Let $c_1, c_2, \dots$ and $c_n$ be the crossings of $D$. 
By Lemma \ref{thHSS}, we can change only $c_i$ by region crossing changes at some regions. 
In terms of matrices, we can create the column vector such that the $i$th element is $1$ and the others are $0$ by a linear combination of some columns of $M$, for any $i \in \{ 1, 2, \dots ,n \}$. 
This means the rank of $M$ is $n$. 
\end{proof}
\phantom{x}

\noindent Then we prove Proposition \ref{eight}. 

\phantom{x}
\noindent {\it Proof of Proposition \ref{eight}.} \ 
Consider a simultaneous equations whose coefficient matrix is $M$. 
Since the degree of freedom of the solution is obtained by subtracting the rank of $M$ from the number of columns of $M$, in this case the degree is $(n+3)-n=3$ by Lemmas \ref{c3} and \ref{fullrank}. 
Since we work on modulo 2, the number of the solutions is $2^3=8$. 
\hfill$\Box$

\phantom{x}
\noindent We remark that Proposition \ref{eight} does not hold for spatial-handcuff-graph diagrams. 
Some crossing changes on a diagram with a cutting edge are not realized by region crossing changes, as shown in Figure 4 in \cite{HSS}.

\section{Ineffective sets} 
\label{ineffective-sets} 

\noindent In this section we consider the diagramatical implications of Proposition \ref{eight}. 
A set $S$ of regions of a diagram is said to be {\it ineffective} when the region crossing changes at all the regions in $S$ do not change the diagram \cite{IS}. 
The following is shown for knots: 

\phantom{x}
\begin{lemma}[\cite{shimizu-rcc}]
Let $D$ be a reduced knot diagram with a checkerboard coloring. 
Then the set $B$ of the black-colored regions of $D$ is ineffective. 
\end{lemma}
\phantom{x}

\noindent For reducible diagrams, we may need a modification to $B$ at some reducible crossings to get ineffective (see Figure 11 in \cite{shimizu-rcc}). 
For spatial $\theta$-curves, we have the following: 

\phantom{x}
\begin{corollary}
Let $G$ be a spatial $\theta$-curve consisting of vertices $v_1, v_2$ and edges $e_1, e_2$ and $e_3$. 
Let $D$ be a reduced diagram of $G$, and let $k^i$ be the knot diagram obtained by removing $e_i$ $(i \in \{ 1,2,3 \} )$. 
The set of regions of $D$ which are black-colored on a checkerboard coloring on $k^i$ is ineffective. 
\label{k-ine}
\end{corollary}
\phantom{x}

\begin{proof}
If we choose such black regions, diagonal two regions are chosen around each crossing of $k^i$, four or no regions are chosen at each crossing of $e_i$ because $e_i$ is ignored on $k^i$, and adjoining two regions are chosen around each crossing between $k^i$ and $e_i$ for the same reason. 
Hence, all the crossings are unchanged by the region crossing changes. 
\end{proof}
\phantom{x}

\noindent For reducible diagrams, we may need a modification at some reducible crossings (see Figure 12 in \cite{HSS}). 
Let $D$ be a reduced diagram of a spatial $\theta$-curve on $\mathbb{R}^2$. 
Let $k^i$ be the knot diagram as mentioned in Corollary \ref{k-ine}. 
Give checkerboard coloring to each $k^i$ so that the outer region is colored white. 
We call the set of regions of $D$ which are black-colored (resp. white-colored) $B^i$ (resp. $W^i$) on the checkerboard coloring to $k^i$. 
We have the following: 

\phantom{x}
\begin{lemma}
The equality $ B^l = \left( B^m \cup B^n \right) \setminus \left( B^m \cap B^n \right)$ holds, where $(l, m, n)$ is a permutation of $(1, 2, 3)$. 
\label{lem-set}
\end{lemma}
\phantom{x}

\begin{proof}
For each region of $k^l$ with the above checkerboard coloring, give the value $1$ (resp. $0$) if the region is colored black (resp. white), for each $l \in \{ 1,2,3 \}$. 
And then, for each region $r_i$ of $D$, give the value $f(r_i)$ which is the sum of the values of the corresponding regions for $k^1$, $k^2$ and $k^3$. 
An example is shown in Figure \ref{ex02}. \\
\begin{figure}[ht]
\begin{center}
\includegraphics[width=100mm]{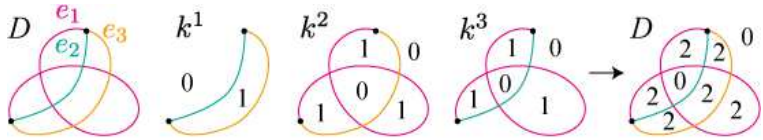}
\caption{Give the checkerboard coloring to each $k^l$ so that the outer region is white, and give the values $1$ and $0$. The labeling $f$ to each region of $D$ is obtained by summing the three values. }
\label{ex02}
\end{center}
\end{figure}

Now we show that each region $r_i$ of $D$ has $f(r_i)=0$ or $2$. 
Let $r_i$ and $r_j$ be regions sharing an edge $e^l$. 
Then $r_i$ and $r_j$ take different values on both $k^m$ and $k^n$ because $e^l$ exists on $k^m$ and $k^n$. 
And they take the same value on $k^l$ because $r_i$ and $r_j$ belongs to the same region on $k^l$. 
Hence, for each pair of regions $r_i$ and $r_j$ sharing an edge, the difference between $f(r_i)$ and $f(r_j)$ is $0$ or $2$. 
Since the outer region $r_0$ has $f(r_0)=0$, and the value of $f$ can be at most $3$, every region $r_i$ has $f(r_i)=0$ or $2$, with the breakdown $0+0+0$ or $0+1+1$. 
Therefore, $B^l$ is obtained by $ \left( B^m \cup B^n \right) \setminus \left( B^m \cap B^n \right)$. 
\end{proof}
\phantom{x}

\noindent Let $S$ be a set of regions of a reduced diagram $D$. 
Let $I$ be an ineffective set for $D$. 
We can obtain the same result of region crossing changes at the regions in $S$ by retaking the regions to $\left( S \cup I \right) \setminus \left( S \cap I \right)$ (see \cite{shimizu-rcc} for knots). 
For a reduced diagram of a spatial $\theta$-curve, by taking $B^1$, $W^1$ and $B^2$ as $I$  and the above retaking, we can obtain eight sets of regions whose effects by region crossing changes are the same. 
We remark that $B^1$, $W^1$ and $B^2$ are independent; 
Looking around a vertex, we can see that neither $B^1$, $W^1$ nor $B^2$ can be obtained by a combination of the others. 
See Figure \ref{aroundV}. 
\begin{figure}[ht]
\begin{center}
\includegraphics[width=60mm]{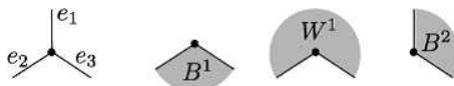}
\caption{Neither $B^1$, $W^1$ nor $B^2$ can be obtained by a combination of the others. }
\label{aroundV}
\end{center}
\end{figure}

\section*{Acknowledgments}

The authors thank Kota Koashi, Atsushi Oya, Hiroaki Saito, Shunta Saito and Mao Totsuka for valuable discussions in the seminars in Gunma College. 
They are also very grateful to Yoshiro Yaguchi for valuable discussions and helpful comments.

\end{document}